\documentclass{tran-l}

\newtheorem{thm}{Theorem}[section]

\newtheorem{lem}[thm]{Lemma}
\newtheorem{prop}[thm]{Proposition}
\theoremstyle{definition}

\theoremstyle{remark}

\newtheorem{exa}[thm]{Example}
\newtheorem{fa}[thm]{Fact}
\numberwithin{equation}{section}

\begin{document}

\title[Action of prime ideals on generalized derivations-I]{Action of prime ideals on generalized derivations-I}
\author{Nadeem ur Rehman$^1$, Hafedh M. Alnoghashi$^2$}

\address{$^1$ Department of Mathematics, Aligarh Muslim University, 202002 Aligarh, India}
\email{rehman100@gmail.com, nu.rehman.mm@amu.ac.in}
\address{$^2$ Department of Mathematics, Faculty of Science, Sana'a University, Yemen}
\email{halnoghashi@gmail.com}

%




\begin{abstract}
In the present paper, we investigate the commutativity of quotient ring $R/P$ where $R$ is any ring and $P$ is a prime ideal of $R$ which admits generalized derivations are satisfying some algebraic identities acting on prime ideals $P$.
\end{abstract}

\maketitle

\vskip 10pt


\section*{Introduction}
Throughout this article, $R$ will represent an associative ring. Recall that a proper ideal $P$ of $R$ is said to be prime if for any $x,y\in R,$ $xRy\subseteq P $ implies that $x\in P$ or $y\in P.$ Therefore, $R$ is called a prime ring if and only if $(0)$ is the prime ideal of $R.$ For any $x,y\in R,$ the symbol $[x,y]$ will denote the commutator $xy-yx,$ while the symbol $x\circ y$ will stand for the anticommutator $xy+yx.$
A mapping $d:R\rightarrow R$ is said to be a derivation of a ring $R$ if $d$ is additive and satisfies $d(xy)=d(x)y+xd(y)$ for all $x,y\in R.$
The very first example of derivation is the mapping $x \mapsto [x,a]$ for all $x \in R$ and $a$ is fixed element of $R$. Such a mapping is called the inner derivation of $R$. More generally, if $d$ is derivation of $R$ and  $F:R\rightarrow R$ be an additive  mapping such that $F(xy)=F(x)y+xd(y)$ for all $x,y\in R$, then $F$ is called a generalized derivation of $R$ with the associated derivation $d$.  For a fixed $a, b \in R$ a typical example of a generalized derivation is the mapping $x \mapsto ax +xb$, which is called the generalized inner derivation induced by $a$ and $b$, with associated derivation $x \mapsto [x, a]$. Further, in a very systematic paper \cite{Lee}, Lee extended the notion of generalized derivation.

During the last few decades there has been an ongoing interest in the study of relationship between the commutative structure of associative rings and certain types of derivations defined on them. In this vein, Daif and Bell \cite{DaifBell} studied derivations of semiprime rings that fix the commutators of appropriate subsets. Precisely, they proved that if $R$ is a semiprime ring, $I$ a nonzero ideal of $R$ and $d$ a derivation of $R$ such that $d([x,y])=[x,y]$ for all $x,y\in I$, then $I$ is contained in $Z(R)$.  Ashraf and first Author \cite{AshrafRehman} examined the same identity on square-closed Lie ideals of prime rings. In \cite{Quadri}, Quadri et al. extended this result to the class of generalized derivations and proved that if $R$ is a prime ring, $I$ a nonzero ideal of $R$ and $R$ admits a generalized derivation $F$ associated with nonzero derivation $d$ such that $F([x,y])=[x,y]$ for all $x,y\in I$, then $R$ is commutative. Recently many authors have obtained commutativity of prime and semiprime rings admitting suitably constrained additive mappings, as automorphisms, generalized derivations acting on appropriate subsets of the rings, see~\cite{mamouni2020derivations}, \cite{mir2020commutativity}, and \cite{TSD}.

In the present paper, we aim to investigate the commutativity of quotient ring $R/P$ where $R$ is any ring and $P$ is a prime ideal of $R$ which admits generalized derivations are satisfying algebraic identities acting on prime ideals $P$. Moreover, some examples are given to demonstrate that the condition imposed on the hypotheses of various theorems is essential.

\section{The Main Result}
The proof of the following fact is that a group cannot written as the set-theoretic union of its two proper subsets.
\begin{fa}\label{f1}
Let $R$ be a ring, $P$ be a prime ideal of $R$ and $S$ an additive subgroup of $R$. Let $\phi: S \to R$ and $\xi: S \to R$ be additive functions such that $\phi(s)R\xi(s) \in P$ for all $s \in S$. Then either $\phi(s) \in P$ for all $s \in P$ or $\xi(s) \in P$ for all $s\in S$.
\end{fa}
\begin{lem}\label{lem-1}
Let $R$ be a ring and $P$ be a prime ideal of  $R$, If
\begin{itemize}
  \item[(i)]$[x,y]\in P$
  \item[(ii)]$x\circ y\in P$
\end{itemize}
for all $x,y\in R$, then $R/P$ is a commutative integral domain.
\end{lem}

\begin{proof}
$(i)$ Assume that $[x,y]\in P$ and $\overline{R}=R/P,$ then $\overline{[x,y]}=\overline{0}$ and so $[\overline{x},\overline{y}]=\overline{0}$ hence $\overline{R}$ is a commutative, and since $P$ is a prime ideal of  $R,$ then $\overline{R}$ is an integral domain. Thus $\overline{R}$ is a commutative integral domain.

\noindent
$(ii)$ Assume that $x\circ y\in P.$ Replacing $y$ by $yr$ in last relation and using it, where $r\in R,$ we have
$y[x,r]\in P$ and since $P\neq R,$ then $[x,r]\in P,$ and by $(i)$, then $R/P$ is a commutative integral domain.
\end{proof}


\begin{prop}\label{lem-a-1}
Let $R$ be a ring, $P$ is a prime ideal of  $R$. If $R$ admits a generalized derivation $F$ with associated  derivation $d$ satisfying  $[x, F(x)]\in P$ for all $x \in R,$ then  either $R/P$ is a commutative integral domain or $d(R)\subseteq P$.
\end{prop}

\begin{proof}
 We have
\begin{align}\label{a-1.1}
  [x,F(x)]\in P
\end{align}
for all $x \in R.$ By linearizing (\ref{a-1.1}), we have
\begin{align}\label{a-1.2}
  [x,F(y)]+[y,F(x)]\in P
\end{align}
for all $x,y \in R.$ Replacing $y$ by $yx$ in (\ref{a-1.2}) and applying it and (\ref{a-1.1}), we get
\begin{align}\label{a-1.3}
  [x,y]d(x)+y[x,d(x)]\in P
\end{align}
for all $x,y \in R.$ Left multiplying (\ref{a-1.3}) by $r,$ where $r\in R,$ we obtain
\begin{align}\label{a-1.4}
  r[x,y]d(x)+ry[x,d(x)]\in P
\end{align}
for all $x,y,r \in R.$ Putting $ry$ instead of $y$ in (\ref{a-1.3}) and using (\ref{a-1.3}) one can see that
\begin{align}\label{a-1.5}
  r[x,y]d(x)+[x,r]yd(x)+ry[x,d(x)]\in P
\end{align}
for all $x,y,r \in R.$ Subtracting (\ref{a-1.4}) from (\ref{a-1.5}), we thereby obtaining
$[x,r]yd(x)\in P$ this implies that $[x,r]Rd(x)\subseteq P.$ Thus, by Fact \ref{f1} either $[x,r]\in P$ or $d(x)\in P,$ for all $x, r \in R$. In the first case, $R/P$ is a commutative integral domain by Lemma \ref{lem-1}$(i)$. In the second case, $d(R)\subseteq P$.
\end{proof}


\begin{thm}\label{thm-1}
Let $R$ be a ring, $P$ is a prime ideal of  $R$. If $R$ admits a generalized derivation $F$ with associated  derivation $d$  satisfying   any one of the following conditions:
\begin{enumerate}
\item[$(i)$] $F(xy) \pm F(x)F(y)\in P$ for all $x,y\in R$
\item[$(ii)$] $F(xy)\pm F(y)F(x)\in P$ for all $x,y\in R$
\end{enumerate}
then either $R/P$ is a commutative integral domain or $d(R)\subseteq P$.
\end{thm}

\begin{proof} $(i)$
 First we consider the case
\begin{align}\label{1.1}
  F(xy)-F(x)F(y)\in P
\end{align}
for all $x,y\in R.$ Replacing $y$ by $yr$ in (\ref{1.1}) and using it, where $r\in R,$ we have
\begin{equation}\label{e1}
xyd(r) - F(x) yd(r) \in P ~\text{for all} ~x, r \in R.
\end{equation}
Now, replacing $x$ by $xz$ in equation (\ref{e1}), we get
\begin{equation}\label{e2}
xzyd(r) - F(x)zyd(r) - xd(z)yd(r) \in P
\end{equation}
for all $x, y, z, r \in R$. Again, replace $y$ by $zy$ in equation (\ref{e1}), to get
\begin{equation}\label{e3}
xzyd(r) - F(x)zyd(r) \in p
\end{equation}
for all $x, y, z, r \in R$. Now comparing (\ref{e2}) and (\ref{e3}), we find that $xd(z)yd(r) \in P$ for all $x, y, z, r \in R$ and hence $[x, s]d(z)yd(r) \in P$ for all $x, y, z, r, s \in R$, that is, $[x, s]d(z)Rd(r) \subseteq P$. Therefore, either $[x, s]d(z) \in P$ or $d(r) \in P$. In the first case, if  $[x, s]d(z)\in P$ for all $x, z, s \in R$, then $[x, s] R d(z) \subseteq P$ and hence by the primeness of $P$ and since we get either $ [x, s] \in P$ or $d(z) \in P$ for all $x, z, s \in R$. If $d(z) \in P$, then $d(R) \subseteq P$. On the other hand, if $ [x, s] \in P$ then   by Lemma  \ref{lem-1}$(i)$, $R/P$ is a commutative integral domain.\\

By the similar approach, we can prove the same conclusion holds for    $F(xy)+F(x)F(y)\in P$ for all $x, y \in R$.\\

\noindent
$(ii)$  Assume that
\begin{align}\label{2.1}
  F(xy)-F(y)F(x)\in P
\end{align}
for all $x,y\in R.$ Replacing $x$ by $xy$ in (\ref{2.1}) and using it, we have
\begin{align}\label{2.2}
  xyd(y)-F(y)xd(y)\in P
\end{align}
for all $x,y\in R.$ Writing $tx$ instead of $x$ in (\ref{2.2}), where $t\in R,$ we get
\begin{align}\label{2.3}
  txyd(y)-F(y)txd(y)\in P
\end{align}
for all $x,y,t\in R.$ Left multiplying (\ref{2.2}) by $t,$ where $t\in R,$ we obtain
\begin{align}\label{2.4}
  txyd(y)-tF(y)xd(y)\in P
\end{align}
for all $x,y,t\in R.$ Subtracting (\ref{2.3}) from (\ref{2.4}), this gives $[F(y),t]xd(y)\in P$ that is $[F(y),t]Rd(y)\subseteq P.$
Since $P$ is a prime ideal of $R$ and applying the Fact \ref{f1} shows that either $[F(y),t]\in P$ for all $y, t \in R$ or $d(R) \subseteq P$. If $[F(y), t] \in P$ for all $y, t \in R$, then by Lemma \ref{lem-a-1} $d(R)\subseteq P$ or $R/P$ is a commutative integral domain.\\

We may obtain the same conclusion by the same argument, when  $F(xy)+ F(y)F(x)\in P$ for all $x,y\in R$.

\end{proof}



\begin{thm}\label{thm-3}
Let $R$ be a ring, $P$ is a prime ideal of  $R$. If $R$ admits a generalized derivation $F$ with associated  derivation $d$  satisfying   any one of the following conditions:
\begin{enumerate}
\item[$(i)$] $F(xy)\pm xy\in P$ for all $x,y\in R$
\item[$(ii)$] $F(xy) \pm yx \in P$ for all $x,y\in R$
\item[$(iii)$]  $F(x)F(y) \pm xy \in P$ for all $x,y\in R$
\item[$(iv)$] $F(x)F(y) \pm yx \in P$ for all $x,y\in R$
\end{enumerate}
then either $R/P$ is a commutative integral domain or $d(R)\subseteq P$.

\end{thm}

\begin{proof}$(i)$
We have
\begin{align}\label{3.1}
F(xy)-xy\in P
\end{align}
for all $x,y\in R.$
If $F=0$, then $xy \in P$ and hence $[x, r] y \in P$ for all $x, y, r \in R$. Since $P \neq R$ and $P$ is prime so $R/P$ is a commutative integral domain. Now, onward we assume that $F\neq 0$.
Replacing $y$ by $yr$ in (\ref{3.1}) and using it, where $r\in R,$ we have $xyd(r)\in P$, that is, $[x, z]yd(r) \in P$ for all $x, y, z, r \in R$.
Since $P\neq R$, and $P$ is prime. Therefore, either $[x, z] \in P$ for all $x, z \in R$  or  $d(R)\subseteq P$. If $[x, z] \in P$ for all $x, z\in R$, then by Lemma \ref{lem-1}$(i)$, $R/P$ is a commutative integral domain.\\

If $F(xy)-xy\in P$ for all $x, y \in R$, the generalized derivation $-F$ satisfies the condition $(-F)(xy)-xy\in P$ for all $x, y \in R$, and hence by above we get the required result.\\

\noindent
$(ii)$ By our hypothesis
\begin{align}\label{4.1}
F(xy)-yx\in P
\end{align}
for all $x,y\in R.$ If $F=0$, then $yx\in P$ for all $x, y \in R$ and hence $[y, r]x \in P$. Since $P \neq R$ and $P$ is prime, so by Lemma \ref{lem-1}$(i)$ $R/P$ is a commutative integral domain. Now, we assume that $F \neq 0$.
Replacing $x$ by $xy$ in (\ref{4.1}) and using it, where $r\in R,$ we have $xyd(y)\in P$ and since $P\neq R,$ then
\begin{align}\label{4.2}
yd(y)\in P
\end{align}
for all $y\in R.$ By linearizing (\ref{4.2}), we get
\begin{align}\label{4.3}
xd(y)+yd(x)\in P
\end{align}
for all $x,y\in R.$ Writing $ry$ instead of $y$ in (\ref{4.3}), where $r\in R,$ we obtain
\begin{align}\label{4.4}
xd(r)y+xrd(y)+ryd(x)\in P
\end{align}
for all $x,y,r\in R.$ Left multiplying (\ref{4.3}) by $r,$ where $r\in R,$ this gives
\begin{align}\label{4.5}
rxd(y)+ryd(x)\in P
\end{align}
for all $x,y,r\in R.$ Subtracting (\ref{4.5}) from (\ref{4.4}), we have $xd(r)y+[x,r]d(y)\in P.$ Putting $xr$ instead of $x$ in last relation and using (\ref{4.2}), we get $[x,r]rd(y)\in P.$ Replacing $x$ by $xs$ in last relation and using it, we have $[x,r]srd(y)\in P,$ that is $[x,r]Rrd(y)\subseteq P$. Thus, by Fact \ref{f1}, we have either $[x, r] \in P$ or $rd(y) \in P$ for all $x, r, y \in R$. If $[x, r] \in P$ for all $x, r \in R$, then by Lemma \ref{lem-1}$(i)$, $R/P$ is a commutative integral domain. if $rd(y) \in P$, for all $r, y \in R$, since $P\neq R$ then $d(y)\in P$,  that is $d(R)\subseteq P.$\\

If $F(xy)+ yx\in P$ for all $x, y \in R$, then by similar arguments as above with necessary variations yields the required conclusion.\\

\noindent
$(iii)$ First we have
\begin{align}\label{7.1}
F(x)F(y) - xy\in P
\end{align}
for all $x,y\in R.$ If $F = 0$, then $xy \in P$ and hence $[x, r]y\in P$ for all $x, y, r \in R$. Since $P$ is prime so, $R/P$ is a commutative integral domain by Lemma \ref{lem-1}$(i)$. Onward we assume that $F \neq 0$.
 Replacing $y$ by $yr$ in (\ref{7.1}) and using it, where $r\in R,$ we have $F(x)yd(r)\in P$ that is $F(x)Rd(r)\subseteq P.$ Then $F(x)\in P$ or $d(r)\in P.$ If $F(x)\in P,$ then $F(R)\subseteq P.$ By using last relation in  (\ref{7.1}), we get $xy\in P$ that is, $[x, r]y \in P$ for all $x, y r \in P$. Since $P\neq R$ and $P$ is prime we get  $[x, r]\in P$ for all $x, r \in R$ and hence  $R/P$ is a commutative integral domain. On the other hand $d(R)\subseteq P$.\\

Using a similar technique with necessary variation, we can prove the same conclusion holds for $F(x)F(y) + xy\in P$ for all $x, y \in R$.\\

 \noindent
 $(iv)$
 Assume that
\begin{align}\label{8.1}
F(x)F(y) -yx\in P
\end{align}
for all $x,y\in R.$ If $F=0$, then $yx \in P$ using the same techniques as used in the proof of $(ii)$,  $R/P$ is a commutative integral domain. Now, onward we assume that $F \neq 0$.
Replacing $y$ by $yx$ in (\ref{8.1}) and using it, we have $F(x)yd(x)\in P$ that is $F(x)Rd(x)\subseteq P.$ Since $P$ is a prime ideal of $R,$ then by Fact \ref{f1}, either $F(R)\subseteq  P$ or $d(R)\subseteq  P$. If  then $F(R)\subseteq P$ and by using last relation in (\ref{8.1}), we get $yx\in P$ and hence  using similar arguments as used above we get $R/P$ is a commutative integral domain.

If $F(x)F(y  + yx\in P$ for all $x, y \in R$, then using a similar approach, the result follows.

\end{proof}



\begin{thm}\label{thm-5}
Let $R$ be a ring, $P$ is a prime ideal of  $R$. If $R$ admits a generalized derivation $F$ with associated  derivation $d$  satisfying   any one of the following conditions:
\begin{enumerate}
\item[$(i)$] $F(xy)- [x,y]\in P$ for all $x,y\in R$
\item[$(ii)$] $F(xy)-  (x\circ y) \in P$ for all $x,y\in R$
\item[$(iii)$] $F(x)F(y)-[x,y]\in P$ for all $x,y\in R$
\item[$(iv)$] $F(x)F(y) - (x\circ y) \in P$ for all $x,y\in R$
\end{enumerate}
then $R/P$ is a commutative integral domain.

\end{thm}

\begin{proof} $(i)$
Assume that
\begin{align}\label{5.1}
F(xy) - [x,y]\in P
\end{align}
for all $x,y\in R.$ If $F=0$, then $[x, y]\in P$ and hence by Lemma \ref{lem-1}$(i)$, $R/P$ is a commutative integral domain. On ward we assume that $F \neq 0$.
Now,replacing $y$ by $yr$ in (\ref{5.1}) and using it, where $r\in R$, we have
\begin{align}\label{5.2}
xyd(r) - y[x,r]\in P
\end{align}
for all $x,y,r\in R.$ Writing $ty$ instead of $y$ in (\ref{5.2}), where $t\in R$, we get
\begin{align}\label{5.3}
xtyd(r) - ty[x,r]\in P
\end{align}
for all $x,y,r,t \in R.$ Left multiplying (\ref{5.2}) by $t,$ where $t\in R,$ we obtain
\begin{align}\label{5.4}
txyd(r) - ty[x,r]\in P
\end{align}
for all $x,y,r,t\in R.$ Subtracting (\ref{5.3}) from (\ref{5.4}), this gives $[x,t]yd(r)\in P$ that is $[x,t]Rd(r)\subseteq P.$ Thus, either  $[x,t]\in P$ or $d(r)\in P.$ If $[x,t]\in P$, then $R/P$ is a commutative integral domain by Lemma~\ref{lem-1}(i). In case $d(R)\subseteq P.$ By using last relation in (\ref{5.2}), we have $y[x,r]\in P$ and since $P\neq R,$ then $[x,r]\in P$ and hence $R/P$ is a commutative integral domain by Lemma~\ref{lem-1}(i). \\

\noindent
$(ii)$
Assume that
\begin{align}\label{6.1}
F(xy) - (x\circ y)\in P
\end{align}
for all $x,y\in R.$ If $F=0$, then $x\circ y\in P$ and hence by Lemma \ref{lem-1}$(ii)$, we get the required result. Now onward we assume that $F \neq 0$.
Replacing $y$ by $yr$ in (\ref{6.1}) and using it, where $r\in R,$ we have
\begin{align}\label{6.2}
xyd(r)+ y[x,r]\in P
\end{align}
for all $x,y,r\in R.$ Now, using the same arguments as used in $(i)$ after  (\ref{5.2}), we will have $R/P$ is a commutative integral domain.\\

\noindent
$(iii)$ Assume that
\begin{align}\label{9.1}
F(x)F(y) - [x,y]\in P
\end{align}
for all $x,y\in R.$ If $F=0$, then $[x,y]\in P$ for all $x, y \in R$ and hence by Lemma \ref{lem-1}$(i)$, we get the required result. Onward we assume that $F \neq 0$.
Replacing $y$ by $yr$ in (\ref{9.1}) and using it, we have
\begin{align}\label{9.2}
F(x)yd(r)- y[x,r]\in P
\end{align}
for all $x,y,r\in R.$ Putting $r=x$ in (\ref{9.2}), we get $F(x)yd(x)\in P$ that is $F(x)Rd(x)\subseteq P.$ Thus by Fact \ref{f1} either $F(x)\in P$ or $d(x)\in P$ for all $x \in R$. If  $F(x)\in P$ for all $x \in R$ then $F(R)\subseteq P$ and hence  by using last relation in (\ref{9.1}), we get $[x,y]\in P$ for all $x, y \in R$ and hence  by Lemma~\ref{lem-1}(i) $R/P$ is a commutative integral domain. In the second case $d(R)\subseteq P$ and by using last relation in (\ref{9.2}), then $y[x,r]\in P$ and since $P\neq R,$ then $[x,r]\in P$ for all $x, r \in R$  and hence $R/P$ is a commutative integral domain by  Lemma~\ref{lem-1}(i).\\

\noindent
$(iv)$
Assume that
\begin{align}\label{10.1}
F(x)F(y) - (x\circ y)\in P
\end{align}
for all $x,y\in R.$ If $F =0$, then $x\circ y\in P$ for all $x, y \in R$, then $R/P$ is a commutative integral domain by Lemma \ref{lem-1}$(ii)$. Now we assume that $F \neq 0$
Replacing $y$ by $yr$ in (\ref{9.1}) and using it, we have
\begin{align}\label{10.2}
F(x)yd(r)+ y[x,r]\in P
\end{align}
for all $x,y,r\in R.$ Now, applying the similar techniques as used after  (\ref{9.2}) in the proof of  $(iii)$ yields the required result.
\end{proof}

 By similar arguments as above with necessary variations, we can prove the following:
\begin{thm}\label{thm-6}
Let $R$ be a ring, $P$ is a prime ideal of  $R$. If $R$ admits a generalized derivation $F$ with associated  derivation $d$  satisfying   any one of the following conditions:
\begin{enumerate}
\item[$(i)$] $F(xy)+ [x,y]\in P$ for all $x,y\in R$
\item[$(ii)$] $F(xy)+ (x\circ y) \in P$ for all $x,y\in R$
\item[$(iii)$] $F(x)F(y)+[x,y]\in P$ for all $x,y\in R$
\item[$(iv)$] $F(x)F(y) + (x\circ y) \in P$ for all $x,y\in R$
\end{enumerate}
then $R/P$ is a commutative integral domain.

\end{thm}
\section{Examples}

In this section we construct some examples to show that the prime ideal
in the hypothesis of  our results are essential.

\begin{exa}\label{ex1}
  Consider the ring $R=\left \{\left(\begin{matrix}
0 & a & b \\
0 & 0 & 2c \\
0 & 0 & 0
\end{matrix}\right): a, b, c\in\mathbb{Z}_4 \right \}.$
Let \linebreak $P=\left\{\left(\begin{matrix}
0 & 0 & 0 \\
0 & 0 & 0 \\
0 & 0 & 0
\end{matrix}\right)\right\}$ be an ideal of $R.$
Let us define $F, d: R \to R$ by
$$F\left(\begin{matrix}
0 & a & b \\
0 & 0 & 2c \\
0 & 0 & 0
\end{matrix}\right)=\left(\begin{matrix}
0 & 2a & 0 \\
0 & 0 & 0 \\
0 & 0 & 0
\end{matrix}\right),
  d\left(\begin{matrix}
0 & a & b \\
0 & 0 & 2c \\
0 & 0 & 0
\end{matrix}\right)=\left(\begin{matrix}
0 & 0 & c \\
0 & 0 & 0 \\
0 & 0 & 0
\end{matrix}\right).$$ It is easy to very that $F$ is a generalized derivation associated with derivation $d$  $(i)$ $F(xy) \pm F(x)F(y)\in P$ and $(ii)$ $F(xy) \pm F(y)F(x)\in P$ for all $x, y \in R$.  Since
$\left(\begin{matrix}
0 & 1 & 0 \\
0 & 0 & 0 \\
0 & 0 & 0
\end{matrix}\right)R\left(\begin{matrix}
0 & 1 & 0 \\
0 & 0 & 0 \\
0 & 0 & 0
\end{matrix}\right)\subseteq P,$ but $\left(\begin{matrix}
0 & 1 & 0 \\
0 & 0 & 0 \\
0 & 0 & 0
\end{matrix}\right)\not\in P$ is not a prime ideal of $R$.
We see that $d(R)\not\subseteq P$ and $R/P$ is noncommutative. Hence the prime ideal is crucial for Theorem \ref{thm-1}$(i)$  and $(ii)$.

\end{exa}


\begin{exa}
Consider the ring $R$, ideal $P$ and derivation $d$ as in Example \ref{ex1}.
 Let $F:R\rightarrow R$ defined by $F(x)=\pm x$ for all $x \in R$. Then it is straightforward to verify that $F$ is a generalized derivation in $R$.
   Also we see that $P$ is not a prime ideal of $R$ because
$$\left(\begin{matrix}
0 & 1 & 0 \\
0 & 0 & 0 \\
0 & 0 & 0
\end{matrix}\right)R\left(\begin{matrix}
0 & 1 & 0 \\
0 & 0 & 0 \\
0 & 0 & 0
\end{matrix}\right)\subseteq P,  ~~~~\left(\begin{matrix}
0 & 1 & 0 \\
0 & 0 & 0 \\
0 & 0 & 0
\end{matrix}\right)\not\in P.$$
Moreover, it is straightforward to check that $F$ satisfies $(i)$ $F(xy) \pm xy\in P,$ (ii) $F(x)F(y) \pm xy\in P$.
However,  $d(R)\not\subseteq P$ and $R/P$ is noncommutative. Hence, prime ideal is essential in Theorem \ref{thm-3}$(i)$ and $(iii)$.
\end{exa}



\begin{exa}
Consider $S$ be a ring such that $s^2=0$ for all $s\in S,$ but the product of some elements of $S$ is nonzero. Since $s^2=0,$ so $(s+t)^2=0$ for all $s, t\in S$ this implies that $s\circ t=0$ for all $s, t\in S.$ Suppose
$R=\left \{\left(\begin{matrix}
x & y  \\
0 & x
\end{matrix}\right): x, y\in S \right \}.$ Define $F=d:R\rightarrow R$ as $F\left(\begin{matrix}
x & y  \\
0 & x
\end{matrix}\right)=\left(\begin{matrix}
0 & y  \\
0 & 0
\end{matrix}\right),$ clearly $F$ and $d$ are  (generalized) derivations. Let $P=\left \{\left(\begin{matrix}
0 & 3y  \\
0 & 0
\end{matrix}\right): y\in S\right \}.$ We see that $\left(\begin{matrix}
0 & y  \\
0 & 0
\end{matrix}\right)R\left(\begin{matrix}
3x & 0  \\
0 & 3x
\end{matrix}\right)\subseteq P,$ but $\left(\begin{matrix}
0 & y  \\
0 & 0
\end{matrix}\right),\left(\begin{matrix}
3x & 0  \\
0 & 3x
\end{matrix}\right)\not\in P$ so $P$ is not prime ideal of $R.$ Also we see that $F(x)F(y)+(x\circ y) \in P$ and $F(x)F(y)-(x\circ y) \in P,$ but $d(R)\not\subseteq P$ and $R/P$ is noncommutative. Therefore, prime ideal is essential in the hypothesis of Theorem \ref{thm-5}$(iv)$  and Theorem \ref{thm-6}$(iv)$.

\end{exa}

\subsection*{Acknowledgment}
For the First author, this research is supported by the National
Board of Higher Mathematics (NBHM), India, Grant No. $02011/16/2020$ NBHM (R. P.)
R \& D II/ $7786$.


\end{document}